\documentclass{amsart}

\usepackage{amssymb}

\newif\iffurther
\furthertrue

\newtheorem{thm}{Theorem}[section] 

\newtheorem{cor}[thm]{Corollary}

\newtheorem{lem}[thm]{Lemma}
\newtheorem{prop}[thm]{Proposition}

\def\g{\mathfrak{g}}

\def\Dim{{\operatorname{Dim}}}
\def\Dimsup{\underline{{\operatorname{Dim}}}}

\def\Ker{{\operatorname{Ker}}}

\def\ad{{\operatorname{ad}}}

\begin{document}

\title{Growth of finitely generated simple Lie algebras}

\author{Be'eri Greenfeld}
\address{Ramat-Gan 5290002, Israel}
\email{beeri.greenfeld@gmail.com}

\thanks{The author is grateful to the referee for his/her useful comments on the paper.}

\begin{abstract}
We realize any submultiplicative increasing function which is equivalent to a polynomial proportion of itself as the growth function of a finitely generated simple Lie algebra. Thus, we obtain a huge and diverse source of intermediate growth functions of simple Lie algebras.
As an application, we resolve two open problems posed by Petrogradsky on Lie algebras with intermediate growth.
\end{abstract}

\maketitle

\section{Introduction}\label{sec:intro}

Let $A$ be a finitely generated algebra. Pick a finite dimensional vector space $V$ which generates the algebra over the base field $F$, and consider the sequence:

$$\gamma_{A,V}(n) =  \dim_F \left(V+V^2+\cdots +V^n\right) $$
where $V^n$ is the space spanned by all $n$-fold products of elements from $V$. For example, in the associative unital case one can choose $1\in V$ and so $V^n\subseteq V^{n+1}$.

This function encodes important combinatorial, algebraic and geometric data of the algebra under investigation. In fact, this function is independent of choice of $V$ up to a natural equivalence relation $\sim$ (to be explicitly defined in Section \ref{sec:2}); hence, from now on, we think of $\gamma_A(n)$ as the growth function of $A$, knowing it is well defined only up to equivalence.

The problem of describing how growth functions of finitely generated associative algebras look like has inspired a flurry of interesting research articles; see for instance \cite{BellZelmanov, BBL, BGJalg, BGIsr, BartholdiSmoktunowicz}. For more on growth of algebras see \cite{KrauseLenagan}.

The behavior of growth functions of Lie 
algebras is rather different compared to growth of associative algebras, and seems, from certain points of view, even more enigmatic. For instance, they need not obey Bergman's gap theorem (which asserts that the growth of an associative algebra cannot be super-linear and subquadratic; an extreme case is demonstrated in \cite{Petro2020}, where restricted Lie algebras whose growth is oscillating between quasilinear and almost exponential are constructed); polynomial growth rate for complex (finitely) $\mathbb{Z}$-graded simple Lie algebras is very rigid, in the sense that it forces one of several concrete structures (by Matiheu's classification, solving Kac's conjecture \cite{Mathieu}); there exist examples of nil Lie algebras of polynomial growth for arbitrarily large base fields \cite{ShestakovZelmanov} (the analog for associative algebras is unknown for uncountable fields); etc. 

Interesting connections between growth of Lie algebras and the growth of their universal enveloping algebras were studied in \cite{Petrogradsky96,Petrogradsky00,Petrogradsky97,Petrogradsky93,Smith}. This work motivates our main theme in this note, whose aim is to prove the following:

\begin{thm} \label{thm_lie}
Let $f:\mathbb{N}\rightarrow \mathbb{N}$ be an increasing, submultiplicative function such that $f(n)\sim n f(n)$. Then there exists a finitely generated simple Lie algebra $\g$ whose growth function satisfies: $$\gamma_\g(n)\sim f(n).$$
\end{thm}
As an application, we are able to answer two questions left open by Petrogradsky (Corollaries \ref{answer Petrogradsky1} and \ref{answer Petrogradsky2}).
Petrogradsky defined a hierarchy of dimensions, $\Dim^q$ and $\Dimsup^q$ for $q=1,2,\dots$ which can be associated with any function.
He constructed Lie and associative algebras of interesting growth types, including non-integral $q$-dimensions, and left as an open problem \cite[Remark in page 347]{Petrogradsky00} and \cite[Remark in page 651]{Petrogradsky97} the question of whether for every $q=3,4,\dots$ and every $\sigma\in (0,1)$ there exists a Lie algebra $\g$ with $\Dim^q(\gamma_{\g}(n))=\Dimsup^q(\gamma_{\g}(n))=\sigma$. We answer this in the affirmative in Corollary \ref{answer Petrogradsky1}.

Petrogradsky also observed that there exist functions `lying between' the layers (namely, $\Dim^q(\gamma_\mathfrak{g}(n))=\infty$ but $\Dim^{q+1}(\gamma_\mathfrak{g}(n))=0$) and left open \cite[Page~464]{Petrogradsky96} the question whether there exist Lie algebras lying between the layers. As another byproduct of our work, we are able to construct Lie algebras whose growth functions lie between $\Dim^q$ and $\Dim^{q+1}$ for each $q\in \mathbb{N}$ (Corollary \ref{answer Petrogradsky2}), thereby answering this problem as well.

We remark that Petrogradsky's questions do not involve simple Lie algebras, and our solutions to them can be derived, in fact, from Lemma \ref{lemma commutator}; however, we feel that Theorem \ref{thm_lie} is of its own interest, and after proving it we are able to provide constructions solving Petrogradsky's questions within the class of \textit{simple} Lie algebras. We remark that simple groups of intermediate growth were only recently constructed (by Nekrashevych, see \cite{SimpInte}).

The resulting Lie algebras of Theorem \ref{thm_lie} are $\mathbb{Z}$-graded (but have infinite dimensional homogeneous components). Note that the condition that $f(n)\sim nf(n)$ is satisfied by any `sufficiently regular' growth function which is rapid at least as $n^{\ln n}$ (as also mentioned in \cite{BartholdiSmoktunowicz}; however, pathological counterexamples exists).

Our construction combines ingredients of four flavors:
\begin{itemize}
    \item Methods developed in \cite{BGJalg} and slightly improved here for controlling the complexity growth of uniformly recurrent words;
    \item Theory of convolution algebras of \'etale groupoid algebras of Bernoulli shifts developed by Nekrashevych in \cite{Nekr};
    \item Finite generation results by Alahmadi, Alsulami, Jain and Zelmanov from \cite{AA, AAJZ};
    \item Classical result of Baxter \cite{Baxter} on simplicity of commutator Lie algebras modulo its intersection with the center.
\end{itemize}

\section{Growth of Lie algebras}\label{sec:2}

Let $F$ be an arbitrary base field (resp.~of characteristic $\neq 2$) and let $A$ be a finitely generated associative (resp.~Lie) $F$-algebra. Recall that the growth function of $A$ with respect to a generating subspace (resp.~Lie generating subspace) is: $$\gamma_{A,V}(n)=\dim_F \left(\sum_{i=0}^{n} V^i\right)$$ or, if $A$ is a Lie algebra: $$\gamma_{A,V}(n)=\dim_F \left(\sum_{i=0}^{n} V^{[i]}\right)$$
where $V^{[i]}$ is spanned by all length-$i$ Lie products of elements from $V$.
This function is independent of choice of $V$ up to the following equivalence relation: we say that $f\preceq g$ if there exist some $C,D>0$ such that $f(n)\leq Cg(Dn)$ for all $n\in \mathbb{N}$, and $f\sim g$ if $f\preceq g$ and $g\preceq f$.

We begin with the following construction, enabling us to transfer from growth functions of certain monomial algebras to growth functions of finitely generated Lie algebras. Given an associative algebra $A$ let $[A,A]$ be the Lie algebra spanned by all (additive) commutators in $A$.

\begin{lem} \label{lemma commutator}
Let $A$ be a finitely generated monomial algebra generated by two nilpotent elements. Assume that the center of $Q$ is trivial, namely, $Z(A)=F$. Then the Lie algebra $\mathfrak{A}=[A,A]$ is finitely generated and:
$$\gamma_A(n)\sim \gamma_\mathfrak{A}(n).$$
\end{lem}

\begin{proof}
Let $A=F\left<x_1,x_2\right>/I$ where $I\triangleleft F\left<x_1,x_2\right>$ is a monomial ideal containing $x_1^t,x_2^t$ for some $t\in \mathbb{N}$. Write: $A=\bigoplus_{n=0}^{\infty}A(n)$, the homogeneous decomposition with respect to the length-induced grading on $A$.
We claim that for all $n>2$ we have that $$\dim_F \left(\mathfrak{A}\cap A(n)\right)\geq \frac{1}{4}\dim_F A(n-2).$$
Fix $n>2$ and write $A(n-1)=x_1S_1\oplus x_2S_2$ where $S_{1,2}\subseteq A(n-2)$, and observe that for some $i\in\{1,2\}$ we have: $$\dim_F x_iS_i\geq \frac{1}{2}\dim_F A(n-1).$$
Now $\Ker(\ad_{x_1})\cap\Ker(\ad_{x_2})=Z(A)=F$ has zero intersection with $A(m)$ for $m>0$ so for some $j\in \{1,2\}$ we have that $$\dim_F\left(\Ker(\ad_{x_j})\cap x_iS_i\right)\leq \frac{1}{2}\dim_F x_iS_i,$$
It follows that:
$$\dim_F \left(\mathfrak{A}\cap A(n)\right)\geq \dim_F\ad_{x_j}(x_iS_i)\geq \frac{1}{2}\dim_F x_iS_i\geq \frac{1}{4}\dim_F A(n-2).$$
To complete the proof, notice that by \cite{AA}, the Lie algebra $\mathfrak{A}$ is finitely generated, say, by some finite dimensional homogeneous space $V$. Then for all $k\in \mathbb{N}$ we have that $\mathfrak{A}\cap A(k)\subseteq \sum_{m=0}^{\infty}V^{[m]}$ so applying the $k$-th degree projection $T_k:\mathfrak{A} \twoheadrightarrow A(k)$ for arbitrary $n\geq k$ we observe that $$\mathfrak{A}\cap A(k)=T_k(\mathfrak{A}\cap A(k))\subseteq T_k(\sum_{m=0}^{n}V^{[m]})$$
so: $$\gamma_{\mathfrak{A}}(n) = \dim_F \sum_{m=0}^{n}V^{[m]}\geq \dim_F\bigoplus_{m=3}^{n}\left(\mathfrak{A}\cap A(m)\right)\geq \frac{1}{4}\dim_F \bigoplus_{m=1}^{n-2}A(m)\sim\gamma_A(n).$$

On the other hand, it is evident that $$\gamma_{\mathfrak{A}}(n) = \dim_F \sum_{m=0}^{n}V^{[m]}\leq \dim_F \bigoplus_{m=0}^{dn} A(m)\sim \gamma_A(n)$$
where $d$ is the maximum degree of elements from $V$. Hence $\gamma_A(n)\sim \gamma_{\mathfrak{A}}(n)$, as desired.
\end{proof}

\section{Combinatorics of infinite words} \label{sec:mon}

Suppose $\Sigma=\{x_1,\dots,x_d\}$ is a finite alphabet. An infinite word $w\in \Sigma^\mathbb{N}$ is \textit{recurrent} if every factor appears infinitely many times, and \textit{uniformly recurrent} if for every factor $u$ of $w$ there exists a constant $C=C(u)$ such that every factor (namely, a finite subword) of $w$ of length $C$ contains a copy of $u$. The \textit{complexity} function of $w$ is: $$c_w(n):=\#\{\text{length-}n\  \text{factors of }w\}.$$

The monomial algebra \textit{associated with} $w$, denoted $A_w$ is the quotient of the free algebra $F\left<\Sigma\right>$ modulo the ideal generated by all monomials which do not appear as factors of $w$. Thus, $A_w$ is spanned by monomials which are factors of $w$, and so the growth of $A_w$ satisfies $\gamma_{A_w}(n)= c_w(1)+\cdots+c_w(n)$. It is known that $A_w$ is prime if and only if $w$ is recurrent \cite[Theorem~3.22~(a)]{BBL}, and it is just-infinite if and only if $w$ is uniformly recurrent \cite[Theorem~3.2.2~(2)]{Madill}. For more on monomial algebras, see \cite{BBL, Madill}.
We need the following technical lemma for the sequel:

\begin{lem} \label{lem_monomial}
Suppose $w$ is an infinite word in $\Sigma^{\mathbb{N}}$. Then there exists an infinite word $w'\in \{x,y\}^{\mathbb{N}}$ with $c_w(n)\sim c_{w'}(n)$ such that if moreover $w$ is recurrent (resp.~uniformly recurrent) then so is $w'$.
\end{lem}

\begin{proof}
Consider the substitution $\sigma:\Sigma^{*}\rightarrow \{x,y\}^{*}$ given by $\sigma:x_i\mapsto xy^{i}$. Then $\sigma$ substitutes $w$ into an infinite word in $\{x,y\}^{\mathbb{N}}$, say, $w'$. It is straightforward to check that $w'$ is recurrent (resp.~uniformly recurrent) whenever $w$ is.

To calculate the complexity growth of $w'$, notice that: $$c_w(n)\leq c_{w'}((d+1)n)$$
since $\sigma$ carries length-$n$ factors of $w$ to factors of $w'$ of length $\leq (d+1)n$. For the opposite direction, fix a factor $v'$ of $w'$.
It is possible to extend $v'$ with $y^i$ from the right and with $xy^j$ ($0\leq i,j\leq d$) from the left such that the resulting word has the form $\sigma(v)$ where $v$ is a factor of $w$. This process gives rise to a map of monomials: $$\Psi: \{\text{length-}n\ \text{factors of }w'\}\rightarrow \bigcup_{p=0}^{2d+2} \{\text{length-}(n+p)\ \text{factors of }w\}$$ such that a factor of $w$ has at most $(d+1)^2$ preimages ($d+1$ possibilities to delete letters from each side). Thus: $$c_{w'}(n)\leq (d+1)^2\sum_{p=0}^{2d+2} c_w(n+p),$$
 and therefore $c_w(n)\sim c_{w'}(n)$.
\end{proof}

We need the following improved version of \cite[Theorem~2.1]{BGJalg}:

\begin{thm} \label{simple}
Let $g:\mathbb{N}\rightarrow \mathbb{N}$ be an increasing, submultiplicative function such that $g(n)\sim ng(n)$. Then there exists a uniformly recurrent word $w$ on some finite alphabet such that $c_w(n)\sim g(n)$.
\end{thm}

\begin{proof}
We define a new function, $f(n)=ng(n)$. It is evident that $f\sim g$, and $f$ is increasing and submultiplicative as long as $g$ is.
We do not repeat the construction from \cite{BGJalg} but rather explain the adjustment we made in the formulation of the result. Indeed, the proof of \cite[Theorem~2.1]{BGJalg} produces a monomial algebra of the form $A_w$ for some uniformly recurrent word $w$ over a suitable finite alphabet.  
Next modify $f$ as follows. Pick $t\geq 1$ such that $f(2^tn)\geq nf(n)$ and put:
\[
f'(n):=\sum_{i=0}^{t-1} n^{-\frac{i}{t}} f(2^i n)
\]
Then $f'$ is increasing and $f\sim f'$, and in addition $f'(2n)\geq \frac{1}{2}n^{\frac{1}{t}}f'(n)$ (see \cite[Section~3.1]{BGIsr}):

\begin{eqnarray*}
f'(2n) & = & \sum_{i=0}^{t-1} (2n)^{-\frac{i}{t}} f(2^{i+1} n) \\
& \geq & \frac{1}{2}\sum_{i=0}^{t-1} n^{-\frac{i-1}{t}} f(2^{i} n) \\
& \geq & \frac{1}{2} n^{\frac{1}{t}}\sum_{i=0}^{t-1} n^{\frac{i}{t}} f(2^i n) \\
& \geq & \frac{1}{2}n^{\frac{1}{t}}\sum_{i=0}^{t-1} n^{-\frac{i}{t}} f(2^i n) + \left(n^{-1}f(2^t n) - f(n)\right) \\
& \geq & \frac{1}{2}n^{\frac{1}{t}}f'(n)
\end{eqnarray*}
Note that $f'$ need not take integral values; however, diving into the proof of \cite[Theorem~2.1]{BGJalg}, we see that we need only assume that the given function is defined on a set of the form $\{2^{tn}\}_{n=1}^{\infty}$ and satisfies $f'(2^{n+1})\leq f'(2^n)^2$ (rather than submultiplicativity); this will follow from the second condition which now follows.
We now show that $f'(n)$ fulfills the conditions of \cite[Theorem~2.1]{BGJalg}:

\begin{enumerate}
    \item The first condition of \cite[Theorem~2.1]{BGJalg} (namely, that $nf'(n)\leq f'(\alpha n)$ for some $\alpha>0$) is satisfied since $f'(n)\sim f(n)\sim g(n)$ and $g(n)\sim ng(n)$.
    
    \smallskip
    
    \item The second condition (namely, that for all $\beta>0$ there exists some $n_\beta$ such that for all $n\geq n_\beta$ we have that $f'(2^{n+1})\leq \frac{1}{\beta}f'(2^n)^2$) is satisfied since:

    \begin{eqnarray*}
    f'(2^{n})^2 & = & \left(\sum_{i=0}^{t-1} 2^{-\frac{in}{t}} f(2^{n+i}) \right)^2 \\
    & \geq & \sum_{i=0}^{t-1} 2^{-\frac{2in}{t}} f(2^{n+i})^2 \\
    & = & \sum_{i=0}^{t-1} 2^{-\frac{2in}{t}}\cdot 2^{2n+2i} g(2^{n+i})^2 \\
    & \geq & \sum_{i=0}^{t-1} 2^{-\frac{2in}{t}+2n+2i} g(2^{n+i+1}) 
     \end{eqnarray*}
but: 
\begin{eqnarray*}
-\frac{2in}{t}+2n+2i & = & n\left(2-\frac{2i}{t}\right)+2i\\
& \geq & n\left(\frac{1}{2} - \frac{i}{t}\right) + \left(1 - \frac{1}{t} \right)i+n+1 \\
& = & \frac{1}{2} n + \left(-\frac{i(n+1)}{t}+n+i+1\right),
\end{eqnarray*}
so the above yields:

\begin{eqnarray*}
f'(2^n)^2 & \geq & 2^{\frac{1}{2} n} \sum_{i=0}^{t-1} 2^{-\frac{i(n+1)}{t}} \cdot 2^{n+i+1} g(2^{n+i+1}) \\
& \geq & 2^{\frac{1}{2} n}\sum_{i=0}^{t-1} 2^{-\frac{i(n+1)}{t}} f(2^{n+i+1}) = 2^{\frac{1}{2} n}f'(2^{n+1})
\end{eqnarray*}
and hence:
$$\frac{f'(2^{n+1})}{f'(2^n)^2}\leq 2^{-\frac{1}{2} n}\xrightarrow{n\rightarrow \infty} 0.$$

    \smallskip

    \item As for the third condition (namely, that $\prod_{n=0}^{\infty}\left(1+\frac{f'(2^n)}{f'(2^{n+1})}\right)$ is convergent), notice that since $f'(2n)\geq \frac{1}{2}n^{\frac{1}{t}}f'(n)$, the sequence $\Big\{\frac{f'(2^{n})}{f'(2^{n+1})}\Big\}_{n=1}^{\infty}$ decays exponentially and therefore $\prod_{n=0}^{\infty}\left(1+\frac{f'(2^n)}{f'(2^{n+1})}\right)$ is convergent.
\end{enumerate}
Finally, we obtain a uniformly recurrent word $w$ whose complexity function satisfies: 
\[
c_w(n)\sim f'(n)\sim f(n)\sim g(n). \qedhere
\]
\end{proof}
We need the following well known auxiliary lemma, whose proof is included for the reader's convenience:

\begin{lem} \label{auxiliary}
If $w\in \Sigma^\mathbb{N}$ is a recurrent word then there exists a bi-infinite word, $u\in \Sigma^{\mathbb{Z}}$ such that $w$ and $w'$ share the same factors.
\end{lem}

\begin{proof}
We denote by $w[i,j]$ the subword of $w$ starting at the $i$-th slot and ending at the $j$-th slot. Let $u_1=w[1,1]$, then by recurrence there exists $l_2\geq 2$ such that $w[l_2,l_2]=u_1$. Let $u_2 = w[1, 2l_2-1]$; it can be written as $u_2 = p_2u_1q_2$ (where $p_2 = w[1,l_2-1]$, $q_2 = w[l_2+1, 2l_2-1]$). Note that $|p_2|=|q_2|\geq 1$.

Now suppose $u_t$ was constructed according to $l_2,\dots,l_t$. Denote $c=|u_t|$. Then there exists $l_{t+1}\gg c$ such that $w[l_{t+1},l_{t+1}+c-1]=u_t$. Set $u_{t+1} = w[1, 2l_{t+1}+c-2]$. Then $u_{t+1} = p_{t+1}u_tq_{t+1}$ (where $p_{t+1}=w[1,l_{t+1}-1]$, $q_{t+1}=w[l_{t+1}+c,2l_{t+1}+c-2]$, and in particular $|p_{t+1}|=|q_{t+1}|\geq 1$). 

This way we have a bi-infinite word $u = \lim_{t\rightarrow \infty} u_t$. Obviously every subword of $u$ is a subword of $w$, and also vice versa, since every finite subword of $w$ is contained in some $w[1,m]$ for $m\gg 1$.
\end{proof}

\section{Convolution algebras of \'{e}tale groupoids} \label{groupoids}

A key ingredient in our construction is convolution algebras of \'{e}tale groupoids. For a comprehensive exposition of the topic, we refer the reader to \cite{Nekr}. For the reader's convenience, we give a brief introduction of the topic now, focusing on the examples and results which will be needed in the next sections.

A \textit{groupoid} $\mathfrak{G}$ is a set endowed with a partial multiplication operation, which is associative and invertible (whenever defined). Equivalently, we can think of $\mathfrak{G}$ as the set of isomorphisms of a small category (a small category is a category in which the classes of objects and morphisms are sets).

We will be using \textit{topological groupoids}; these are groupoids endowed with topologies, with respect to which multiplication and inversion are continuous.
A \textit{bisection} $F\subseteq \mathfrak{G}$ is a subset for which the maps $g\mapsto g^{-1}g,\ g\mapsto gg^{-1}$ (called \textit{origin} and \textit{target}, respectively) induce homeomorphisms (onto the images of $F$ under them); elements of the form $g^{-1}g$ are called \textit{units}. A topological groupoid is \textit{\'etale} if its topology admits a basis of neighborhoods consisting of bisections.

Let $\mathfrak{G}$ be a Hausdorff \'etale groupoid whose space of units is totally disconnected. Then one can associate with $\mathfrak{G}$ an associative algebra, $F[\mathfrak{G}]$, called the \textit{convolution algebra} of $\mathfrak{G}$. This algebra consists of all continuous, compactly supported functions $f:\mathfrak{G}\rightarrow F$ (here the base field $F$ is arbitrary and endowed with the discrete topology). While the linear structure of $F[\mathfrak{G}]$ is clear, its multiplication is given by convolution:
$$(f_1\cdot f_2)(g)=\sum_{h\in \mathfrak{G}} f_1(h)f_2(h^{-1}g).$$
The algebra $F[\mathfrak{G}]$ is spanned by characteristic functions of compact open bisections.

A leading example from which we derive our constructions now follows.
Fix a finite alphabet $\Sigma$ and consider the space of bi-inifnite sequences over $\Sigma$, namely, $\Sigma^{\mathbb{Z}}$ (endowed with the usual Tychonoff topology). For a closed, shift-invariant subset $\mathfrak{X}\subseteq \Sigma^{\mathbb{Z}}$ one defines the \textit{groupoid of germs} $\mathfrak{G}_{\mathfrak{X}}$ of the action $\mathbb{Z} \curvearrowright \mathfrak{X}$: its elements are of the form $(t,x)\in \mathbb{Z}\times \mathfrak{X}$ and (partial) multiplication is given by:
$$(m_2, x^{m_1})\cdot(m_1,x) = (m_1+m_2,x)$$
where $x^m$ is the $m$-th shift of $x\in \mathfrak{X}$,
and thus the groupoid is a $\mathbb{Z}$-indexed union of countably many copies of $\mathfrak{X}$ (each copy constitutes a compact open bisection). This is an \'etale Hausdorff groupoid whose space of units is totally disconnected.

When $\mathfrak{X}=\mathfrak{X}_w$ is generated by $w\in \Sigma^{\mathbb{Z}}$ (namely, the shift-orbit of $w$ is dense in $\mathfrak{X}$) we denote the resulting groupoid by $\mathfrak{G}_w$. 

When $w$ is uniformly recurrent and aperiodic then $\mathfrak{X}_w$ is minimal (i.e.~has no non-empty proper shift-invariant closed subset) and $F[\mathfrak{G}_w]$ is simple with an effective growth bound as follows from the next result.

\begin{thm}[{Theorem~1.2 from \cite{Nekr}}]
Suppose that $\mathfrak{X}_w$ is minimal and aperiodic. Then the algebra $F[\mathfrak{G}_w]$ is simple, and its growth satisfies:

$$C^{-1}n\cdot c_w(C^{-1}n) \leq \gamma_{F[\mathfrak{G}_w]}(n)\leq Cn\cdot c_w(Cn) $$
for some $C>1$.
\end{thm}
(Here $c_w(n)$ is the complexity function of $w$ discussed in Section \ref{sec:mon}.)
From now on, we further restrict to the case of s subshift of $\{0,1\}^{\mathbb{Z}}$ generated by some uniformly recurrent word $w$, and again denote its closure by $\mathfrak{X}_{w}$. Then: $$F[\mathfrak{G}_{w}]=F[D_x,D_y,T^{\pm 1}]$$ is generated by the characteristic functions of the cylindrical sets $$\{u\in \mathfrak{X}_{w}|\ u(0)=0\},\ \{u\in \mathfrak{X}_{w}|\ u(0)=1\}$$ associated with $0$ and $1$ (namely, $D_x$ and $D_y$, respectively) and the characteristic functions of the sets of germs of the shift and its inverse, $T^{\pm 1}$. Notice that we have a ring homomorphism: $$\phi:A_{w}\rightarrow F[\mathfrak{G}_{w}]$$ given by: $$\phi(x)=D_xT,\ \phi(y)=D_yT.$$ Since $w$ is uniformly recurrent, $A_{w}$ is a just-infinite algebra and $\phi$ is an injective homomorphism.

This example is discussed in \cite{Nekr}, where the following infinite matrix representation of $F[\mathfrak{G}_w]$ is given. While we will not be using this representation directly, we find it useful in order to understand the flavor of this algebra:

$$   D_x = 
\left(
\begin{array}{ccccc}
     \ddots & \ddots & & &  \\
     \ddots & w[-1] & 0 & 0 & \\
     & 0 & w[0] & 0 & \\
     & 0 & 0 & w[1] & \ddots \\
     & & & \ddots & \ddots
\end{array}
\right),\ \ \ \ \ D_y = 
\left(
\begin{array}{ccccc}
     \ddots & \ddots & & &  \\
     \ddots & 1-w[-1] & 0 & 0 & \\
     & 0 & 1-w[0] & 0 & \\
     & 0 & 0 & 1-w[1] & \ddots \\
     & & & \ddots & \ddots
\end{array}
\right) $$

and:

$$   T = 
\left(
\begin{array}{llllll}
     \ddots & \ddots & \ddots & \phantom{\ddots} &  \\
      & 0 & 1 & 0 & \phantom{\ddots} \\
      & & 0 & 1 & 0 &  \phantom{\ddots}\\
      &  &  & 0 & 1 & \ddots \\
      & & & & 0 & \ddots \\
      & & & & & \ddots
\end{array}
\right)$$
We need the following result on monomial algebras in the next section, and take advantage of the above discussion to provide a proof from the point of view of convolution algebras:

\begin{lem}
\label{trivial center}

Let $w\in \Sigma^{\mathbb{Z}}$ be a uniformly recurrent, aperiodic bi-infinite word. Let $A_w$ be its associated monomial algebra over $F$. Then $Z(A_w)=F$

\end{lem}

\begin{proof}
Assume on the contrary that there exists some $f\in Z(A_w)\setminus F$. Since $A_w$ is graded, we may assume that $f$ is homogeneous of positive degree (recall that the center of a graded algebra is homogeneous).
Let $K\supseteq F$ be an arbitrary uncountable field extension; then $f\in Z(A_{w}\otimes_F K)$. Utilizing the injective map: $$\phi: A_w\otimes_F K\rightarrow K[\mathfrak{G}_w]$$ and observing that $K[\mathfrak{G}_w]$ is generated as a $K$-algebra by the image of $\phi$ together with $T^{-1}$  we obtain that $f\in Z(K[\mathfrak{G}_w])$ (notice that $T\in \phi(A_w\otimes_F K)$). But since $K[\mathfrak{G}_w]$ is simple, we have that $f-\lambda$ is invertible in $K[\mathfrak{G}_w]$ for each $\lambda\in K$. Since $K$ is uncountable and $\dim_K K[\mathfrak{G}_w]=\aleph_0$, we get that $f$ is algebraic therein, and hence also in $A_w\otimes_F K$. But since $A_w\otimes_F K$ is graded and $f$ is homogeneous, it follows that $f$ is nilpotent, contradicting its invertibility.
\end{proof}

\section{Growth of simple Lie algebras}

We are now ready to prove Theorem \ref{thm_lie}:

\begin{proof}[Proof of Theorem \ref{thm_lie}]
Let $f:\mathbb{N}\rightarrow \mathbb{N}$ be an increasing, submultiplicative function such that $f(n)\sim nf(n)$. Then by Theorem \ref{simple} there exists a uniformly recurrent word $w$ on some finite alphabet with $c_w(n)\sim f(n)$. By Lemma \ref{lem_monomial} there exists a uniformly recurrent word $w'$ on a two-letter alphabet with $c_{w'}(n)\sim f(n)$. Consider the monomial algebra $A_{w'}$. Then it satisfies the conditions of Lemma \ref{lemma commutator} (since $w'$ is uniformly recurrent the two letters do not have arbitrary powers which are factors, so $A_{w'}$ is generated by two nilpotent elements, and since $w'$ is uniformly recurrent and aperiodic then $Z(A_{w'})=F$ by Lemma \ref{trivial center}). But since $f(n)\sim nf(n)$, we have that:
$$f(n)\leq \gamma_{A_{w'}}(n)\leq nc_{w'}(n)\sim f(n)$$
and by Lemma \ref{lemma commutator}:
$$\gamma_{[A_{w'},A_{w'}]}(n)\sim \gamma_{A_{w'}}(n)\sim f(n).$$
Now observe that since $w'$ is recurrent, we may assume by Lemma \ref{auxiliary} it is \textit{bi-infinite}, namely, its hereditary language is the same as of a uniformly recurrent word in $\{0,1\}^{\mathbb{Z}}$; for simplicity we denote its bi-infinite `twin' by $w'$ as well. Now consider the subshift of $\{0,1\}^{\mathbb{Z}}$ generated by $w'$ and denote its closure by $\mathfrak{X}_{w'}$; this gives rise to an \'etale groupoid $\mathfrak{G}_{w'}$ with convolution algebra $F[\mathfrak{G}_{w'}]$, as done in Section \ref{groupoids}.
Recall from Section \ref{groupoids} that we have an injective ring homomorphism: $$\phi:A_{w'}\rightarrow F[\mathfrak{G}_{w'}]$$ given by: $$\phi(x)=D_xT,\ \phi(y)=D_yT$$
which reduces to an embedding of the associated Lie algebras: $$\tilde{\phi}:[A_{w'},A_{w'}]\cong\frac{[A_{w'},A_{w'}]}{Z(A_{w'})\cap [A_{w'},A_{w'}]}\longrightarrow \frac{[F[\mathfrak{G}_{w'}],F[\mathfrak{G}_{w'}]]}{Z(F[\mathfrak{G}_{w'}])\cap [F[\mathfrak{G}_{w'}],F[\mathfrak{G}_{w'}]]}$$
Denote the range of $\tilde{\phi}$ by $\g$. Since $w'$ is uniformly recurrent, it generates a minimal subshift and so $F[\mathfrak{G}_{w'}]$ is a simple $F$-algebra (by \cite[Theorem~1.2]{Nekr}). 

By Baxter's theorem \cite{Baxter}, $\g$ is a simple Lie algebra. 
We note that $\g$ is finitely generated. Indeed, by \cite[Theorem~1]{AAJZ} if $R$ is a finitely generated associative algebra with an idempotent $e\in R$ such that $ReR=R(1-e)R=R$ then $[R,R]$ is finitely generated as a Lie algebra. Since $F[\mathfrak{G}_{w'}]$ is simple and finitely generated, any idempotent $\neq 0,1$ satisfies the required property; in particular, we can take $e=D_x$.
Now evidently: $$\gamma_\g(n)\preceq \gamma_{F[\mathfrak{G}_{w'}]}(n)\sim nc_{w'}(n)\sim f(n)$$
(the middle inequality follows from the growth statement in \cite[Theorem~1.2]{Nekr}) and since $\tilde{\phi}$ is injective and its domain is equal to $[A_{w'},A_{w'}]$ at most up to a $1$-dimensional subspace then: $$f(n)\sim \gamma_{[A_{w'},A_{w'}]}(n)\preceq \gamma_\g(n),$$
and the claim follows.
\end{proof}

\section{Remarks and applications} \label{sec:3}

Let $\g$ be a finitely generated Lie algebra with $L_i$ monomials of length $i$ modulo monomials of length shorter than $i$. Smith \cite{Smith} established the connection between the growth of $\g$ and the growth of its universal enveloping algebra $U(\g)$ by the following formula:

$$\sum_{n=0}^\infty \lambda_U(n)t^n = \prod_{i=1}^{\infty} \frac{1}{(1-t^i)^{L_i}}$$

where $\lambda_U(n)=\gamma_{U(\g)}(n)-\gamma_{U(\g)}(n-1)$ is the discrete derivative of the growth function of $U(\g)$. Using this, she was able to construct universal enveloping algebras with intermediate growth rate.

Petrogradsky carried this connection several steps further: in \cite{Petrogradsky96,Petrogradsky00,Petrogradsky93}, he studied the following hierarchy of functions for $q=1,2,3,\dots$ and $\alpha\in \mathbb{R}^{+}$:
\begin{eqnarray*}
\Phi_\alpha^{1}(n) & = & \alpha,\\
\Phi_\alpha^{2}(n) & = & n^\alpha,\\
\Phi_\alpha^{3}(n) & = & \exp(n^{\alpha/{(\alpha+1)}}),\\
\Phi_\alpha^{q}(n) &=& \exp\left(\frac{n}{(\ln^{(q-3)} n)^{1/\alpha}}\right)\ \ \ \  \text{      for } q=4,5,\dots
\end{eqnarray*}
Where $\ln^{(1)} n=\ln n,\ \ln^{(q+1)} n=\ln(\ln^{(q)} n)$. Petrogradsky proved interesting connections between the growth of a finitely generated Lie algebra and the growth of its universal enveloping algebra in terms of the above hierarchy.
He defined the following notions of dimensions for growth functions (called \textit{$q$-dimensions}):
\begin{eqnarray*}
\Dim^q f(n) &= & \inf\{\alpha\in\mathbb{R}^{+}|f(n)\leq \Phi_\alpha^q(n)\ \forall n\gg 1\}, \\
\Dimsup^q f(n) &= & \sup\{\alpha\in\mathbb{R}^{+}|f(n)\geq \Phi_\alpha^q(n)\ \forall n\gg 1\}.
\end{eqnarray*}
Petrogradsky constructed Lie and associative algebras of interesting growth types, including non-integral $q$-dimensions; he then left as an open problem \cite[Remark in page 347]{Petrogradsky00} and \cite[Remark in page 651]{Petrogradsky97} whether for $q=3,4,\dots$ every $\sigma\in (0,1)$ is realizable as $\Dim^q(\gamma_\g(n))=\Dimsup^q(\gamma_\g(n))=\sigma$ for some finitely generated Lie algebra.

Using Theorem \ref{thm_lie} we are now able to answer this question in the affirmative.

\begin{cor} \label{answer Petrogradsky1}
For any $q=3,4,\dots$ and $\sigma\in(0,1)$ there exists a finitely generated Lie algebra $\g$ such that $\Dim^q(\gamma_\g(n))=\Dimsup^q(\gamma_\g(n))=\sigma$.
\end{cor}

\begin{proof}
Fix $q,\sigma$. Let $\varphi_\sigma^q(n)=\lceil \Phi_\sigma^q(n)\rceil$. It is straightforward to check that $\varphi_\sigma^q$ is increasing (at least for $n\gg 1$) and submultiplicative. In addition, for $q\geq 4$: $$\varphi_\sigma^q(2n)\geq \exp\left(\frac{n}{(\ln^{(q-3)} n)^{1/\sigma}}+\sqrt{n}\right)=\exp(\sqrt{n})\Phi_\sigma^q(n)\geq n\varphi_\sigma^q(n)$$
and for $q=3$:
$$\varphi_\sigma^3(2n)\geq \exp\left(2^{\sigma/{(\sigma+1)}}n^{\sigma/{(\sigma+1)}}\right)\geq 2^{\frac{\sigma}{{(\sigma+1)}}n^{\sigma/{(\sigma+1)}}}\Phi_\sigma^3(n)\geq n\varphi_\sigma^3(n)$$
for all $n\gg 1$.
Therefore, by Theorem \ref{thm_lie} there exists a finitely generated simple Lie algebra $\g$ whose growth rate is $\gamma_\g(n)\sim \varphi_\sigma^q(n)$. Therefore $\Dim^q(\gamma_\g(n))=\Dimsup^q(\gamma_\g(n))=\sigma$.
\end{proof}

In \cite{Petrogradsky96}, Petrogradsky notes that there exist functions which `lie between' the layers of $q$-dimensions: for instance, he notes that $f(n)=n^{\ln n}$ has that $\Dim^2(f)=\infty$ but $\Dim^3(f)=0$, and leaves open the question whether there exist algebras whose growth functions lie between layers (see page 464, there). Since $f(n)=n^{\ln n}$ is easily seen to satisfy the conditions of Theorem \ref{thm_lie}, it is equivalent to the growth rate of a finitely generated (even simple) Lie algebra, settling this question in the affirmative. In fact:

\begin{cor} \label{answer Petrogradsky2}
For any $q\geq 2$ there exists a finitely generated simple Lie algebra $\g$ such that $\Dim^q(\gamma_{\g}(n))=\infty$ but $\Dim^{q+1}(\gamma_{\g}(n))=0$.
\end{cor}

\begin{proof}
For $q=2$ take an algebra whose growth rate is $\sim n^{\ln n}$, as mentioned in the above paragraph. For $q\geq 3$, consider the function: $$f_q(n)=\exp\left(\frac{n}{(\ln^{(q-2)} n)^{\ln n}}\right)$$
(formally one takes $\lceil f_q(n)\rceil$.) It is straightforward to check that $\Dim^q(f_q(n))=\infty$ but $\Dim^{q+1}(f_q(n))=0$. In addition, it is easy to verify that $f_q(n)$ is increasing and submultiplicative (at least for $n\gg 1$), and also $f_q(2n)\geq nf_q(n)$ for $n\gg 1$. Thus by Theorem \ref{thm_lie} there exists a finitely generated simple Lie algebra $\g$ whose growth rate is $\gamma_{\g}(n)\sim f_q(n)$, proving the claim.
\end{proof}
(Note that since our Lie algebras originate from associateive algebras (namely, $F[\mathfrak{G}_w]$), the same functions are realizable as growth functions of associative algebras.)

We conclude with a remark on a question posed by B.~Plotkin. In \cite[Problem~6]{Plotkin}, he asked if there exists a continuum of non-isomorphic finitely generated simple Lie algebras over a fixed base field. It should be mentioned that in \cite{LichtPass}, Lichtman and Passman answered the analogous question for associative algebras, assuming the base field either contains enough roots of the unity or is countable. Bokut, Chen and Mo proved \cite[Theorem~4.6]{BCM} that an arbitrary countably generated Lie algebra over a countable field is embeddable into a $2$-generated simple Lie algebra (in particular, this answers Plotkin's question on Lie algebras for countable base fields). We remark that \cite[Example~4.3]{Nekr} gives another answer to Plotkin's question for associative algebras (without any assumption on the base field), and using the arguments in Theorem \ref{thm_lie}, to Plotkin's question for Lie algebras over an arbitrary field.


\begin{thebibliography}{99}

\bibitem{AA}
A.~Alahmadi, H.~Alsulami, \textit{Finite generation of Lie derived powers of associative algebras}, Journal of Algebra and its Applications 18 (3) (2019).

\bibitem{AAJZ}
A.~Alahmadi, H.~Alsulami, S.~K.~Jain, E.~Zelmanov, \textit{Finite generation of Lie algebras associated with associative algebras}, J.~Algebra, 426 (2015), 69--78.

\bibitem{Baxter}
W.~E.~Baxter, \textit{Lie simplicity of a special class of associative rings II}, Trans.~AMS.~87 (1) (1958), 63--75.

\bibitem{BellZelmanov}
J.~P.~Bell, E.~Zelmanov, \textit{On the growth of algebras, semigroups, and hereditary languages}, arXiv:1907.01777 [math.RA].

\bibitem{BBL}
A.~Ya.~Belov, V.~V.~Borisenko, V.~N.~Latyshev, \textit{Monomial algebras}, Journal of Mathematical Sciences 87 (3) (1997), 3463--3575.

\bibitem{BCM}
L.~A.~Bokut, Y.~Chen, Q.~Mo, \textit{Gr\"obner-Shirshov bases and embeddings of algebras}, Int.~J.~Alg.~Comp.~20 (7) (2010), 875--900.

\bibitem{BGJalg}
B.~Greenfeld, \textit{Growth of monomial algebras, simple rings and free subalgebras}, J.~Algebra 489 (2017), 427--434.

\bibitem{BGIsr}
B.~Greenfeld, \textit{Prime and primitive algebras with prescribed growth types}, Israel Journal of Mathematics 220 (2017), 161--174.




\bibitem{KrauseLenagan}
G.~Krause, T.~Lenagan, \textit{Growth of algebras and the Gelfand-Kirillov dimension} (revised edition), Graduate studies in mathematics vol.~22, AMS Providence, Rhode Island (2000).

\bibitem{LichtPass}
A.~I.~Lichtman, D.~S.~Passman, \textit{Finitely generated simple algebras: a question of B.~I.~Plotkin}, Isr.~J.~Math.~143 (2004), 341--359.

\bibitem{Madill}
B.~Madill, \textit{Contributions to the theory of
radicals for noncommutative rings}, PhD dissertation, University of Waterloo (2017), available in: uwspace.uwaterloo.ca/handle/10012/12031.

\bibitem{Mathieu}
O.~Mathieu, \textit{Classification of simple graded Lie algebras of finite growth}, Invent.~Math.~108 (1992), 455--589.

\bibitem{Nekr}
V.~Nekrashevych, \textit{Growth of \'{e}tale groupoids and simple algebras}, Int.~J.~Alg.~Comp.~262 (2) (2016), 375–-397.

\bibitem{SimpInte}
V.~Nekrashevych, \textit{Palindromic subshifts and simple periodic groups of intermediate growth}, Annals of Math.~187 (3) (2018), 667--719.

\bibitem{Petrogradsky96}
V.~M.~Petrogradsky, \textit{Intermediate growth in Lie algebras and their enveloping algebras}, Journal of Algebra 179 (1996), 459--482.

\bibitem{Petro2020}
V.~M.~Petrogradsky, \textit{Nil restricted Lie algebras of oscillating intermediate growth}, https://arxiv.org/abs/2004.05157.

\bibitem{Petrogradsky00}
V.~M.~Petrogradsky, \textit{On growth of Lie algebras, generalized partitions, and analytic functions}, Discrete Mathematics 217 (2000), 337--351.

\bibitem{Petrogradsky97}
V.~M.~Petrogradsky, \textit{On Lie Algebras with Nonintegral q-Dimensions}, Proc.~AMS, 125 (3) (1997), 649--656.

\bibitem{Petrogradsky93}
V.~M.~Petrogradsky, \textit{On some types of intermediate growth
in Lie algebras}, Uspekhi Mat.~Nauk 48 5(293) (1993), 181--182.

\bibitem{Plotkin}
B.~Plotkin, \textit{Problems in algebra inspired by universal algebraic geometry}, arXiv:math/0406101 [math.GM].

\bibitem{ShestakovZelmanov}
I.~P.~Shestakov, E.~Zelmanov, \textit{Some examples of nil Lie algebras}, J.~EMS.~10 (2007), 391--398.

\bibitem{Smith}
M.~K.~Smith, \textit{Universal enveloping algebras with subexponenetial but not polynomially
bounded growth}, Proceedings of the American Mathematical Society 60 (1976), 22--24.

\bibitem{BartholdiSmoktunowicz}
A.~Smoktunowicz, L.~Bartholdi, \textit{Images of Golod-Shafarevich algebras with small growth}, Q.~J.~Math. 65 (2) (2014), 421–-438.








\end{thebibliography}
\end{document}